\documentclass[a4paper,11pt]{amsart}

\usepackage{epsfig}
\usepackage{amsthm,amsfonts}
\usepackage{amssymb,graphicx,color}
\usepackage[all]{xy}
\usepackage{verbatim}
\usepackage{hyperref}
\usepackage{enumitem}

\newtheorem{theorem}{Theorem}[section]
\newtheorem*{theorem*}{Theorem}

\newtheorem{corollary}[theorem]{Corollary}
\newtheorem{proposition}[theorem]{Proposition}
\newtheorem{definition}[theorem]{Definition}

\newtheorem{claim}{Claim}[theorem]

\newtheorem{innercustomthm}{Theorem}
\newenvironment{customthm}[1]
  {\renewcommand\theinnercustomthm{#1}\innercustomthm}
  {\endinnercustomthm}

\newcommand{\C}{\mathbb C}
\newcommand{\R}{\mathbb R}

\newcommand{\N}{\mathbb N}

\begin{document}

\title[Metric-geometric Chow theorem]
{Metric-geometric Chow theorem}

\author[J. E. Sampaio]{Jos\'e Edson Sampaio}

\address{Jos\'e Edson Sampaio: Departamento de Matem\'atica, Universidade Federal do Cear\'a,
	      Rua Campus do Pici, s/n, Bloco 914, Pici, 60440-900, 
	      Fortaleza-CE, Brazil. \newline  
              E-mail: {\tt edsonsampaio@mat.ufc.br}
}


\thanks{
The author was partially supported by CNPq-Brazil grant 303375/2025-6 and by the Serrapilheira Institute (grant number Serra -- R-2110-39576).
}

\begin{abstract}
In 2009, Peterzil and Starchenko proved the following beautiful generalization of Chow's
theorem: An entire complex analytic set $X\subset \mathbb{C}^n$ that is definable in an o-minimal structure on $\mathbb{R}$ must be an algebraic set. This result is known nowadays as the o-minimal Chow theorem.
In this article, we present some geometric and metric versions of Chow's theorem that generalize the o-minimal Chow's theorem. For instance, we prove that if $X\subset \mathbb{C}^n$ is a pure $d$-dimensional entire complex analytic set, then the following statements are equivalent:
\begin{enumerate}
 \item $X$ is a complex algebraic set;
 \item $\mathcal{H}^{2d+1}(C(X,\infty))=0$, where $\mathcal{H}^{k}(A)$ denotes the $k$-dimensional Hausdorff measure of $A$, and $C(X,\infty)$ denotes the tangent cone at infinity of $X$;
 \item $\mathcal{H}^{2d+2}(C_{\mathbb{C}}(X,\infty))=0$, where $C_{\mathbb{C}}(X,\infty)$ denotes the complex tangent cone at infinity of $X$;
 \item $\mathcal{H}^{2d}(Z(X,\infty))=0$, where for $A\subset \mathbb{C}^k$, $Z(A,\infty)=\{[v]\in \mathbb{C}P^{k-1};v\in C_{\mathbb{C}}(A,\infty)\}$;
 \item For any $k\in\{d+1,...,n\}$ and for any projection $\pi\colon \mathbb{C}^{n}\to \mathbb{C}^{k}$ such that $\pi^{-1}(0)\cap C_{\mathbb{C}}(X,\infty)=\{0\}$ and $Y=\pi(X)$, $\mathcal{H}^{2d}(Z(Y,\infty))<\mathcal{H}^{2d}(\mathbb{C}P^{d})$.
\end{enumerate}

\end{abstract}

\maketitle

\section{Introduction}

In 1949, Chow proved in \cite{Chow:1949} the famous theorem, currently known as Chow's Theorem, which states that {\it any closed complex analytic subset of a complex projective space is a complex algebraic subset}. Chow's Theorem has many important consequences and versions and is the basis for applying analytic methods in Algebraic Geometry. 
Nowadays, we know that Chow's Theorem can also be proved using the Remmert–Stein Theorem (see \cite{RemmertS:1953}), which states that {\it if $F$ is an analytic set of dimension less than $k$ in some complex manifold $D$, and $M$ is an analytic subset of $D \setminus F$ with all components of dimension at least $k$, then the closure of $M$ is either analytic or contains $F$}.

More recently, in 2009, Peterzil and Starchenko \cite{PeterzilS:2009} proved the following beautiful generalization of Chow's
theorem: {\it An entire complex analytic set $X\subset \mathbb{C}^n$ that is definable in an o-minimal structure on $\mathbb{R}$ must be an algebraic set}. This result is known nowadays as the o-minimal Chow theorem.

Recently, the author presented in \cite{Sampaio:2023} some implications of the study on Lipschitz Geometry at infinity of complex analytic sets for the studies on algebraicity and rigidity of complex analytic sets.
In particular, it was presented in \cite{Sampaio:2023} the following result: {\it Let $X\subset\C^n$ be a pure $d$-dimensional entire complex analytic subset. Then the following statements are equivalent:
\begin{enumerate}
 \item [(1)] $X$ is a complex algebraic set.
 \item [(2)] $C(X,\infty)$ is a $d$-dimensional complex algebraic set;
 \item [(3)] $X$ is bi-Lipschitz homeomorphic at infinity to a complex algebraic set.
\end{enumerate}}

In \cite{SampaioS:2024}, da Silva and the author of this paper proved that a pure $d$-dimensional entire complex analytic subset is an algebraic set if, and only if, it is lipeomorphic at infinity to a definable set in an o-minimal structure on $\R$.

In this article, we prove other characterizations of the algebraicity of analytic sets besides those presented in \cite{Sampaio:2023}, \cite{SampaioS:2024}, and \cite{PeterzilS:2009}.
In order to state our main result, let us define some objects.

\begin{definition}\label{def:realtangentcone}
Let $X\subset \R^m$ be a subset. We say that $v\in \mathbb{R}^{m}$ is {\bf a tangent vector to $X$ at infinity} if there are a sequence of points $\{x_j\}_{j\in \N}\subset X$ and a sequence of real positive numbers $\{t_j\}_{j\in \N}$ such that $\lim\limits _{j\to +\infty }t_j= +\infty$ and $\lim\limits _{j\to +\infty }\frac{1}{t_j}x_j=v$.
We denote by $C(X,\infty)$ the set of all tangent vectors of $X$ at infinity, and we call it {\bf the tangent cone of $X$ at infinity}.
\end{definition}
\begin{definition}\label{complextangentcone}
Let $X\subset \C^m$ be a subset. 
We denote by $C_{\C}(X,\infty)$ the set of all $v\in \mathbb{C}^{m}$ such that there is a sequence of points $\{x_j\}_{j\in \N}\subset X$ and a sequence of numbers $\{t_j\}_{j\in \N}\subset \C\setminus \{0\}$ satisfying $\lim\limits _{j\to +\infty }\|t_j\|= +\infty$ and $\lim\limits _{j\to +\infty }\frac{1}{t_j}x_j=v$. We call $C_{\C}(X,\infty)$ {\bf the complex tangent cone of $X$ at infinity}.
\end{definition}
For an entire complex analytic subset, we define $X\subset\C^n$ $Z(X,\infty)=\{[v]\in \mathbb{C}P^{n-1};v\in C_{\mathbb{C}}(X,\infty)\}$

Here, $\mathcal{H}^{k}(A)$ denotes the $k$-dimensional Hausdorff measure of $A$, and $\dim_H A$ denotes the Hausdorff dimension of $A$. 

Thus, our main result is the following:
\begin{customthm}{\ref*{chow-type-thm}}
Let $X\subset\C^n$ be a pure $d$-dimensional entire complex analytic subset. Then the following statements are equivalent:
\begin{enumerate}
 \item $X$ is a complex algebraic set;
 \item $\mathcal{H}^{2d+1}(C(X,\infty))=0$;
 \item $\mathcal{H}^{2d+2}(C_{\mathbb{C}}(X,\infty))=0$;
 \item $\mathcal{H}^{2d}(Z(X,\infty))=0$;
 \item For any $k\in\{d+1,...,n\}$ and for any projection $\pi\colon \mathbb{C}^{n}\to \mathbb{C}^{k}$ such that $\pi^{-1}(0)\cap C_{\mathbb{C}}(X,\infty)=\{0\}$ and $Y=\pi(X)$, $\mathcal{H}^{2d}(Z(Y,\infty))<\mathcal{H}^{2d}(\mathbb{C}P^{d})$.
\end{enumerate}
\end{customthm}
We note that the Hausdorff measure in $\mathbb{C}P^{n-1}$ is taken with respect to the Fubini-Study distance. Moreover, all the subsets of $\R^n$ (or $\C^n$) are considered to be equipped with the induced Euclidean metric.

\begin{corollary}
  Let $X\subset\C^n$ be a pure $d$-dimensional entire complex analytic subset. Then the following statements are equivalent:
\begin{enumerate}
 \item $X$ is a complex algebraic set;
 \item $\dim_{H}(C(X,\infty))<{2d+1}$;
 \item $\dim_{H} C_{\mathbb{C}}(X,\infty)<2d+2$;
 \item $\dim_{H} Z(X,\infty)<2d$.
\end{enumerate}  
\end{corollary}

Since for a set $X$ that is definable in an o-minimal structure on $\mathbb{R}$, $\dim_H C(X,\infty)\leq \dim_H X$ (see Proposition \ref{selection_lemma}), we obtain the result proved by Peterzil and Starchenko \cite{PeterzilS:2009}:
\begin{corollary}
An entire complex analytic set $X\subset \mathbb{C}^n$ that is definable in an o-minimal structure on $\mathbb{R}$ must be an algebraic set.
\end{corollary}

We note that one day before this paper appeared on ResearchGate, Nguyen posted the interesting paper \cite{Nguyen:2026} on ResearchGate with the following theorem, using different notation:
\begin{theorem}
Let $X\subset\C^n$ be a pure $d$-dimensional entire complex analytic subset. If $\mathcal{H}^{2d}(Z(X,\infty))=0$, then $X$ is a complex algebraic set.
\end{theorem}
This implies that the items (1)-(4) in Theorem \ref{chow-type-thm} are equivalent. Here, we use tools similar to those used in \cite{Nguyen:2026}, but the style, notation, and organization of the paper are different.

\section{Algebraicity of analytic sets}

We have the following characterization of complex algebraic sets:

\begin{theorem}\label{chow-type-thm}
Let $X\subset\C^n$ be a pure $d$-dimensional entire complex analytic subset. Then the following statements are equivalent:
\begin{enumerate}
 \item $X$ is a complex algebraic set;
 \item $\mathcal{H}^{2d+1}(C(X,\infty))=0$;
 \item $\mathcal{H}^{2d+2}(C_{\mathbb{C}}(X,\infty))=0$;
 \item $\mathcal{H}^{2d}(Z(X,\infty))=0$;
 \item For any $k\in\{d+1,...,n\}$ and for any projection $\pi\colon \mathbb{C}^{n}\to \mathbb{C}^{k}$ such that $\pi^{-1}(0)\cap C_{\mathbb{C}}(X,\infty)=\{0\}$ and $Y=\pi(X)$, $\mathcal{H}^{2d}(Z(Y,\infty))<\mathcal{H}^{2d}(\mathbb{C}P^{d})$.
\end{enumerate}
\end{theorem}

\begin{proof}
\noindent $(1) \Longrightarrow (2)$.
Assume that $X$ is a complex algebraic set. Note that every complex algebraic set is definable in an o-minimal structure on $\mathbb{R}$. By Proposition \ref{selection_lemma}, we have \(\mathcal{H}^{2d+1}(C(X,\infty))=0\).

\bigskip

\noindent $(2) \Longrightarrow (3)$.
We assume that $\mathcal{H}^{2d+2}(C(X,\infty))=0$.

Let $a\colon \C^n\to \C^n$ be the mapping given by $a(x)=-x$. Let $E=C(X,\infty)\cup a(C(X,\infty))\cup (i\cdot C(X,\infty))$, where $i\cdot C(X,\infty)=\{ix;x\in C(X,\infty)\}$. Thus,
$$
\mathcal{H}^{2d+1}(E)=0.
$$
Consider the mapping $\varphi\colon \R\times \C^n\to \C^n$ given by $\varphi(t,x)=(t+i)x$. We have $\varphi(\R\times E)=C_{\C}(X,\infty)$. Since $\mathcal{H}^{2d+2}(\R\times E)=0$, and $\varphi$ is locally a Lipschitz mapping, we obtain $\mathcal{H}^{2d+2}(C_{\C}(X,\infty))=0$.

\bigskip

\noindent $(3) \Longrightarrow (4)$. Assume $\mathcal{H}^{2d+2}(C_{\C}(X,\infty))=0$. 

Let $\overline{X}$ be the closure of $X$ in $\C P^n$ and $Z=\overline{X} \cap H_{\infty}$, where $H_{\infty}$ is the hyperplane at infinity, i.e., $H_{\infty}=\{(x_0:x_1:....:x_n)\in \C P^n; x_0=0\}$. With the obvious identification $H_{\infty}\cong \C P^{n-1}$, we have that $Z=Z(X,\infty)=\{[v]\in \mathbb{C}P^{n-1};v\in C_{\mathbb{C}}(X,\infty)\setminus\{0\}\}$.
Let $\pi\colon \C^{n+1}\setminus \{0\}\to \C P^n$ be the mapping given by $\pi(x_0,x_1,...,x_n)=(x_0:x_1:....:x_n)$.
Let $\widetilde{X}=\pi^{-1}(\overline{X})\cup \{0\}$. Consider $\widetilde{Z}\subset \C^n$ satisfying $\{0\}\times \widetilde{Z}=\widetilde{X}\cap (\{0\}\times \C^n)=\pi^{-1}(Z)\cup \{0\}$.
Thus, note that $\mathcal{H}^{2d+2}(\widetilde{Z})=0$ if and only if $\mathcal{H}^{2d}(Z)=0$.

\begin{claim}\label{algebraic_cone}
$C_{\C}(X,\infty)=\widetilde{Z}$.
\end{claim}
\begin{proof}[Proof of Claim \ref{algebraic_cone}]  
Since $C_{\C}(X,\infty)$ and $\widetilde{Z}$ are complex cones, i.e., if $v\in C_{\C}(X,\infty)$ (resp. $v\in \widetilde{Z}$) then $t v\in C_{\C}(X,\infty)$ (resp. $t v\in \widetilde{Z}$) for all $t\in \C$, it is enough to show that $C_{\C}(X,\infty)\cap \mathbb{S}^{2n-1}=\widetilde{Z}\cap \mathbb{S}^{2n-1}$.

Let $v\in C_{\C}(X,\infty)$ such that $\|v\|=1$. Then there exist sequences $\{z_j\}_j\subset X$ and $\{\lambda_j\}_j\subset \C\setminus\{0\}$ such that $\lim\limits_{j\to +\infty} z_j=\infty$ and $\lim\limits_{j\to +\infty} \frac{z_j}{\lambda_j}=v$, for some $\lambda\in \C$ such that $|\lambda|=1$. Thus, for each $j$, let $L_j=\pi^{-1}(1:z_j)\cup \{0\}$ and $w_j\in L_j\cap \mathbb{S}^{2n-1}$. Since $(\frac{1}{\|z_j\|}:\frac{z_j}{\|z_j\|})\to (0:v)$ as $j\to +\infty$, we obtain, taking a subsequence, if necessary, that $w_j\to (0,\lambda v)$ for some $\lambda \in \C$ with $|\lambda|=1$. Since $\widetilde{X}$ is a homogeneous complex algebraic set, $(0,v)\in \widetilde{X}$, and thus $v\in \widetilde{Z}$.
This shows that $C_{\C}(X,\infty)\subset \widetilde{Z}$.

Let $v\in \widetilde{Z} $ such that $\|v\|=1$. Since $\{0\}\times \widetilde{Z}=\pi^{-1}(Z)\cup \{0\}$, there exists a sequence $\{(x_{0j},y_j)\}_j\subset  \widetilde{X} \setminus (\{0\}\times \widetilde{Z})$ which converges to $(0,v)$. By multiplying $(x_{0j}:y_j)$ by $\frac{\|x_{0j}\|}{x_{0j}}$, if necessary, we assume that $x_{0j}$ is a positive real number for all $j$. Therefore, $z_j=\frac{1}{x_{0j}}y_j\in X$ for all $j$ and, taking subsequence, if necessary, we can assume that $\{\frac{z_j}{\|z_j\|}\}_j$ converges and, in this case, $\lim \frac{z_{j}}{\|z_{j}\|}=v$, which shows that $v\in C(X,\infty)$.
\end{proof}
Therefore, $\mathcal{H}^{2d+2}(\widetilde{Z})=0$, which implies $\mathcal{H}^{2d}(Z)=0$. Since $Z=Z(X,\infty)$, we obtain $\mathcal{H}^{2d}(Z(X,\infty))=0$.

\bigskip

\noindent $(4) \Longrightarrow (5)$. Assume $\mathcal{H}^{2d}(Z(X,\infty))=0$.
Then, $\mathcal{H}^{2d+2}(C_{\C}(X,\infty))=0$.

Let $\pi\colon \mathbb{C}^{n}\to \mathbb{C}^{d+1}$ be a projection such that $\pi^{-1}(0)\cap C_{\mathbb{C}}(X,\infty)=\{0\}$ and $Y=\pi(X)$.

\begin{claim}\label{claim:proper_map}
 There are constants $C,R>0$ such that $\|z\|\leq C\|\pi(z)\|$ for all $z\in X\setminus B_R(0)$. In particular, $\pi|_X\colon X\to \C^{k}$ is a proper mapping. 
\end{claim}
\begin{proof}[Proof of the \ref{claim:proper_map}]
If this does not happen, then there is a sequence $\{z_j\}_j\subset X$ such that $\lim\|z_j\|=+\infty$ and $\|z_j\|>j\|\pi(z_j)\|$ for all $j$. By taking subsequence, if necessary, we may assume that $\lim\frac{z_j}{\|z_j\|}=z_0$. Then, $z_0\in \pi^{-1}(0)\cap C(X,\infty)\subset \pi^{-1}(0)\cap C_{\C}(X,\infty)$ and this is a contraction, since $\pi^{-1}(0)\cup C_{\C}(X,\infty)=\{0\}$ and $\|z_0\|=1$. Therefore there are such constants $C,R>0$, and thus $\pi|_X\colon X\to \C^{k}$ is a proper mapping.
\end{proof}

Therefore, $Y$ is an entire analytic set and $\pi(C_{\C}(X,\infty))=C_{\C}(Y,\infty)$. Since $\pi$ is Lipschitz, we obtain that $\mathcal{H}^{2d+2}(C_{\C}(Y,\infty))=0$. Then, $\mathcal{H}^{2d}(Z(Y,\infty))=0$.

Since $\mathcal{H}^{2d}(\mathbb{C}P^{d})>0$, we have $\mathcal{H}^{2d}(Z(Y,\infty))<\mathcal{H}^{2d}(\mathbb{C}P^{d})$.

\bigskip

\noindent $(5) \Longrightarrow (1)$. Assume that for any $k\in\{d+1,...,n\}$ and for any projection $\pi\colon \mathbb{C}^{n}\to \mathbb{C}^{k}$ such that $\pi^{-1}(0)\cap C_{\mathbb{C}}(X,\infty)=\{0\}$ and $Y=\pi(X)$, $\mathcal{H}^{2d}(Z(Y,\infty))<\mathcal{H}^{2d}(\mathbb{C}P^{d})$.

If $d=0$, then $X$ is a finite set and, in particular, it is an algebraic set. Indeed, if $X$ is not a finite set, then $X$ is an unbounded set. Then, $C_{\C}(X,\infty)$ contains at least a complex line, and thus $Z(X,\infty)$ contains at least one point. Then $\mathcal{H}^{0}(Z(X,\infty))\geq 1$, which is a contraction with $\mathcal{H}^{2d}(Z(X,\infty))=0$.

Thus, we may assume that $d\geq 1$, which implies that $X$ is an unbounded set.
Moreover, it is enough to consider the case that $X$ is an irreducible set.

Since we are assuming that $d\geq 1$, then $n\geq 2$.

Moreover, if $n-d=1$, then $\mathcal{H}^{2d}(H_{\infty})=\mathcal{H}^{2d}(\mathbb{C}P^{d})>0$. Since $\mathcal{H}^{2d}(Z(X,\infty))<\mathcal{H}^{2d}(\mathbb{C}P^{d})$, we have $Z(X,\infty)\subsetneq H_{\infty}$.
By the Remmert-Stein Theorem (see \cite[Theorem 4.6]{Shiffman:1971}), $\overline{X}$ is an analytic subset of $\C P^n$. By Chow's Theorem, $\overline{X}$ is an algebraic set. Therefore, $X$ is an algebraic set.

In particular, we showed that if $n\leq 2$, then $X$ is an algebraic set.

We may assume that $n\geq 3$ and $n-d\geq 2$.

We are going to proceed by induction on $n$. Assume that for any pure dimensional entire complex analytic subset $Y\subsetneq \C^{n-1}$ such that for any $k\in\{d+1,...,n-1\}$ and for any projection $\tilde\pi\colon \mathbb{C}^{n-1}\to \mathbb{C}^{k}$ such that $\tilde\pi^{-1}(0)\cap C_{\mathbb{C}}(X,\infty)=\{0\}$ and $W=\tilde\pi(Y)$, $\mathcal{H}^{2d}(Z(W,\infty))<\mathcal{H}^{2d}(\mathbb{C}P^{d})$, we have that $Y$ is an algebraic set.

Since $X\subset \C^n$ and $Z=Z(X,\infty)\subsetneq H_{\infty}$.
There is $[p]\in  H_{\infty}\setminus Z$. Let $\pi\colon \C^n\to \C^{n-1}$ be the orthogonal projection such that $\pi^{-1}(0)=[p]$. In particular, $\pi^{-1}(0)\cap C_{\C}(X,\infty)=\{0\}$.

By making a unitary change of coordinates, if necessary, we assume that $\pi$ is the projection onto $n-1$ of the first coordinates.

Let $Y=\pi(X)$. 
Then, by \ref{claim:proper_map}, $Y\subsetneq \C^{n-1}$ is an entire analytic set and $\dim Y=d$. Moreover, $\mathcal{H}^{2d}(Z(Y,\infty))<\mathcal{H}^{2d}(\mathbb{C}P^{d})$.

For any $k\in\{d+1,...,n-1\}$ and for any projection $\tilde\pi\colon \mathbb{C}^{n-1}\to \mathbb{C}^{k}$ such that $\tilde \pi^{-1}(0)\cap C_{\mathbb{C}}(Y,\infty)=\{0\}$ and $W=\tilde\pi(Y)$.

Let $p=\tilde\pi\circ \pi\colon \mathbb{C}^{n}\to \mathbb{C}^{k}$. We have that $p^{-1}(0)\cap C_{\mathbb{C}}(X,\infty)=\{0\}$ and $W=p(X)$.
By the hypotheses, $\mathcal{H}^{2d}(Z(W,\infty))<\mathcal{H}^{2d}(\mathbb{C}P^{d})$. 

By the hypothesis of induction, $Y$ is an algebraic set.

By Rudin's criterion (see \cite{Rudin:1968}), $Y$ is contained, after some unitary change of coordinates, in a domain 
$D=\{(x',x'')\in\C^{n-1}; \|x''\| \leq M(1+ \|x'\|)^s\}$, where $M$ and $s$ are positive constants. However, it follows from the proof of Claim \ref{claim:proper_map} that $X\subset \{(x,t)\in\C^{n}; \|x\| \leq \tilde C(1+ |t|)\}$, for some constant $\tilde C\geq C$.

\begin{eqnarray*}
    \|(x'',t)\| & = & \|x''\|+|t|\\
    & \leq & \|x''\|+\widetilde{C}(1+\|x'\|+M(1+\|x'\|)^s)\\
    & \leq & \overline{C}(1+\|x'\|)^{\widetilde{s}}.
\end{eqnarray*}

Therefore, $X\subset \{(x',x'',t)\in\C^{n}; \|(x'',t)\| \leq \bar{C}(1+ \|x'\|)^{s+1}\}$ for $\bar{C}=\sqrt{2}(\tilde C+M)$. By Rudin's criterion again, $X$ is an algebraic set.

\end{proof}

We obtain the following direct consequence for hypersurfaces:
\begin{corollary}
Let $X\subset\C^n$ be a pure $(n-1)$-dimensional entire complex analytic subset. Then the following statements are equivalent:
\begin{enumerate}
 \item $X$ is a complex algebraic set;
 \item $\mathcal{H}^{2n-1}(C(X,\infty))=0$;
 \item $\mathcal{H}^{2n}(C_{\mathbb{C}}(X,\infty))=0$;
 \item $\mathcal{H}^{2n-2}(Z(X,\infty))=0$;
 \item $\mathcal{H}^{2n-2}(Z(X,\infty))<\mathcal{H}^{2d}(\mathbb{C}P^{d})$.
\end{enumerate}
\end{corollary}

\appendix
\section{Characterizing tangent cones of definable sets}\label{sec:appendix}
An important way to view the tangent cone at infinity of definable sets in an o-minimal structure is through the velocity of curves at infinity. 
In order to learn about o-minimal structures, see \cite{Coste:1999} and \cite{Dries:1998}.

\begin{proposition}\label{selection_lemma}
Let $Z\subset \R^n$ be an unbounded definable set in an o-minimal structure on $\R$. Then $C(Z,\infty)$ is definable, $\dim C(Z,\infty)\leq \dim Z$, and 
$$
C(Z,\infty)=\{w\in \R^n; \exists\ \gamma\in C^0((\varepsilon ,+\infty ), Z) \,\,  s.t.\,\,\gamma(t)=tw+o_{\infty }(t)\},
$$ 
where $g(t)=o_{\infty }(t)$ means $\lim\limits _{t\to +\infty }\frac{g(t)}{t}=0$ and $C^0((\varepsilon ,+\infty ), Z)$ is the set of all continuous functions from $(\varepsilon ,+\infty )$ to $Z$. 
\end{proposition}
\begin{proof}
Suppose that $w\in\R^n$ is a tangent vector of $Z$ at infinity. Let us consider the semialgebraic mapping $\iota\colon\R^n\setminus\{0\}\to \R^n\setminus\{0\}$ given by $\iota(x)=\frac{x}{\|x\|^2}$ and denote $X=\iota(Z\setminus \{0\})$.  Since $Z$ is an unbounded set, the origin is a non-isolated point of $\overline{X}$. Let $\rho\colon \mathbb{S}^{n-1}\times [0,+\infty )\to \R^n\setminus \{0\}$ be the mapping given by $\rho(x,t)=tx$. We see that $\rho|_{\mathbb{S}^{n-1}\times (0,+\infty )}:\mathbb{S}^{n-1}\times (0,+\infty )\to \R^n\setminus \{0\}$ is a semialgebraic homeomorphism with inverse mapping $\rho^{-1}\colon\R^n\setminus \{0\}\to \mathbb{S}^{n-1}\times (0,+\infty )$ given by $\rho^{-1}(x)=(\frac{x}{\|x\|},\|x\|)$. Therefore, the sets $Y=\rho^{-1}(X)\subset \mathbb{S}^{n-1}\times [0,+\infty )$ and $\overline{Y}$ are definable sets in an o-minimal structure on $\R$. We are going to consider two cases:

\bigskip

\noindent 1) Case $w\not=0$. Since $w$ is a tangent vector of $Z$ at infinity, there are sequences $\{s_k\}_{k\in \N}$ of positive real numbers and $\{z_k\}_{k\in \N}\subset Z$ such that $\lim\limits _{k\to +\infty }\|z_k\|=+\infty $ and $\lim\limits _{k\to +\infty }\frac{1}{s_k}z_k=w$. Thus, for each $k\in \N$, let us define $x_k=\iota(z_k)$. In this case,  $v:=\lim\limits _{k\to \infty } s_k\cdot x_k= \frac{w}{\|w\|^2}$. In particular, $\lim\limits _{k\to \infty } \frac{x_k}{\|x_k\|}=\frac{w}{\|w\|}=\frac{v}{\|v\|}$ and $u=(\frac{v}{\|v\|},0)\in \overline{Y}$. Then, by the Curve Selection Lemma, there exists a continuous curve $\beta:[0,\delta)\to \overline{Y}$ such that $\beta(0)=u$ and $\beta((0,\delta ))\subset Y$. By writing $\beta(t)=(x(t),s(t))$, we obtain that $s:[0,\delta )\to \R$ is a definable and non-constant function such that $s(0)=0$ and $s(t)>0$ if $t\in (0,\delta )$. By the Monotonicity Lemma, one can suppose that $s$ is $C^1$ and  strictly increasing in the domain $(0,\delta)$. Hence, $s\colon [0,\delta/2]\to [0,\delta' ]$ is a homeomorphism, where $\delta' =s(\frac{\delta}{2})$. Let us define $\alpha\colon [0,r)\to \overline{Y}$ by
$$
\alpha(t)=\rho\circ \beta\circ s^{-1}(t\|v\|)=\rho(x(s^{-1}(t\|v\|)),s(s^{-1}(t\|v\|)))=t\|v\|x(s^{-1}(t\|v\|)),
$$
where $r=\min\{\frac{\delta'}{\|v\|},\delta'\}$. Therefore,
$$
\lim\limits _{t\to 0^+}\frac{\alpha(t)}{t}=\lim\limits _{t\to 0^+}\frac{t\|v\|x(s^{-1}(t))}{t}=\lim\limits _{t\to 0^+}\|v\|x(s^{-1}(t))=\|v\|x(0)=v,
$$
and thus $\alpha(t)=tv+o(t)$. Finally, by defining $\gamma\colon(\frac{1}{r},+\infty )\to Z$ in the following way $\gamma(t)=\iota^{-1}(\alpha(\frac{1}{t}))$, we obtain
\begin{eqnarray*}
\gamma(t) &=& \frac{\frac{1}{t}v+o(\frac{1}{t})}{\|\frac{1}{t}v+o(\frac{1}{t})\|^2}=t\textstyle{\frac{v}{\|v\|^2}}+o_{\infty }(t)\\
		  &=& tw+o_{\infty }(t).
\end{eqnarray*}
Since $\gamma$ is a composition of continuous mappings,  $\gamma$ is a continuous mapping as well.

\bigskip

\noindent 2) Case $w=0$. In this case,  let $\{x_k\}_{k\in\N}\subset Z$ be a sequence such that  $\lim\limits _{k\to +\infty }\|x_k\|=+\infty $ (this sequence exists because $Z$ is unbounded). Thus, $\{\frac{x_k}{\|x_k\|}\}_{k\in\N}$ is, up to taking a subsequence, a convergent sequence. Let $v\in  \R^n$ be the limit of this sequence, i.e., $\lim\limits _{k\to \infty }\frac{x_k}{\|x_k\|}=v$. Likewise, as was done in Case 1, one can show that there exists a continuous curve $\gamma\colon (\varepsilon,+\infty )\to Z$ such that $\gamma(t)=tv+o_{\infty }(t)$. Let us define $\widetilde\gamma\colon (\varepsilon^{2}, +\infty )\to Z$ by $\widetilde{\gamma}(t)=\gamma(t^{\frac{1}{2}})$. Thus, we have $\widetilde{\gamma}(t)=o_{\infty }(t)=tw+o_{\infty }(t)$.

\bigskip

Reciprocally, if there exists a continuous curve $\gamma:(\varepsilon,+\infty )\to Z$ such that $\lim\limits _{t\to +\infty }|\gamma(t)|=+\infty$ and $\gamma(t)=tw+o_{\infty }(t)\}$, then for each $k\in \N$ we define $s_k=\varepsilon+k+1$ and $z_k=\gamma(s_k)$. Thus, it is clear that $w$ is a tangent vector of $Z$ at infinity, since $\lim\limits _{k\to +\infty }\|z_k\|=+\infty $ and $\lim\limits _{k\to +\infty }\frac{1}{s_k}z_k=w$.

Therefore, 
$$
C(Z,\infty)=\{w\in \R^n; \exists\ \gamma\in C^0((\varepsilon ,+\infty ), Z) \,\,  s.t.\,\,\gamma(t)=tw+o_{\infty }(t)\}.
$$ 
Moreover, we have also proved that $(C(Z,\infty)\cap \mathbb{S}^{n-1})\times \{0\}\subset \overline{Y}\setminus Y$. Then 
$$
\dim(C(Z,\infty)\cap \mathbb{S}^{n-1})\leq \dim(\overline{Y}\setminus Y)<\dim Y=\dim Z,
$$
which implies $\dim(C(Z,\infty))\leq \dim Z$.
\end{proof}

\noindent{\bf Acknowledgements}. 
The author would like to thank Lev Birbrair, Alexandre Fernandes, Vinicius Prado, Euripedes da Silva, and Joserlan da Silva for their interest in this work.


\end{document}